\documentclass[12pt,reqno]{amsart}
\usepackage{amsmath, amsthm, amscd, amsfonts, amssymb, graphicx}
\setbox0=\hbox{$+$}
\newdimen\plusheight
\plusheight=\ht0
\def\+{\;\lower\plusheight\hbox{$+$}\;}

\setbox0=\hbox{$-$}
\newdimen\minusheight
\minusheight=\ht0
\def\-{\;\lower\minusheight\hbox{$-$}\;}

\setbox0=\hbox{$\cdots$}
\newdimen\cdotsheight
\cdotsheight=\plusheight
\def\cds{\lower\cdotsheight\hbox{$\cdots$}}

\renewcommand{\(}{\left\(}
\renewcommand{\)}{\right\)}

\renewcommand{\pmod}[1]{\,(\textup{mod}\,#1)}
\numberwithin{equation}{section}
\theoremstyle{plain}
\newtheorem{theorem}{Theorem}[section]
\newtheorem{lemma}[theorem]{Lemma}

\newtheorem*{theorem*}{Theorem}

\begin{document}
  
 \begin{center}
	{\bf{Some New Congruences Modulo Powers of 2 For $(j,k)$-Regular Overpartition
	}}
	\vspace{.5cm}
	\\Riyajur Rahman and  Nipen Saikia$^\ast$ 
	
	\vspace{.2cm}
	
	Department of Mathematics, Rajiv Gandhi University,\\ Rono Hills, Doimukh, Arunachal Pradesh, India, Pin-791112.\\
	Email(s): riyajurrahman@gmail.com; nipennak@yahoo.com\\
	$^\ast$ \textit{Corresponding author.}
\end{center}
\vskip 5mm
\noindent \textbf{Abstract:} Let  $\overline{p}_{j,k}(n)$ denotes the number of $(j,k)$-regular overpartitions of a positive integer $n$  such that none of the parts is congruent to $j$ modulo $k$. Naika et. al. (2021) proved infinite families of congruences modulo powers of 2 for $\overline{p}_{3,6}(n)$, $\overline{p}_{5,10}(n)$ and $\overline{p}_{9,18}(n)$. In this paper, we obtain infinite  families of congruences modulo power of 2 for $\overline{p}_{4,8}(n)$, $\overline{p}_{6,12}(n)$ and  $\overline{p}_{8,16}(n)$. For example, we prove that, for all integers $n\geq 0$ and $\alpha\geq 0$,  
$$
\overline{p}_{4,8}\left(  5^{2\alpha+1}\left(  16(5n+j)+14\right)  \right) q^n\equiv 0\pmod{64}; \qquad
j=1,2,3,4.$$


\section{Introduction}
A partition of a positive integer $n$ is a non-increasing sequence of positive integers called parts, whose sum is equal to $n$.  The number of partitions of a non-negative integer $n$ is usually  denoted by $p(n)$ (with  $p(0)=1$) and the generating function is given by  
\begin{equation}\label{i1}\sum_{n=0}^{\infty}p(n)q^n=\dfrac{1}{(q;q)_{\infty}},\end{equation}where, for any complex number $a$, 
\begin{equation}\label{i2}
	(a;q)_{\infty}=\prod_{n=0}^{\infty}(1-aq^n), \quad |q|<1.
\end{equation} 

We will use the notation, for any positive integer $k$,
\begin{equation}\label{fff}
f_k:=(q^k; q^k)_\infty.
  \end{equation}
 An overpartition of a non-negative integer $n$ is a partition of $n$ in which the first occurrence of each parts may be overlined. For example, there are 14 overpartition of 4, namely
$$4,\quad \overline{4},\quad3+1,\quad\overline{3}+1,\quad 3+\overline{1},\quad\overline{3}+\overline{1},\quad2+2,\quad\overline{2}+2,\quad2+1+1,\quad\overline{2}+1+1,\quad2+\overline{1}+1,$$$$\quad\overline{2}+\overline{1}+1,\quad1+1+1+1,\quad\overline{1}+1+1+1.
$$If   $\overline{p}(n)$ denotes the number of overpartition of $n$, then the generating function of $\overline{p}(n)$ is given by
 \begin{equation}
  \sum_{n=0}^{\infty}\overline{p}(n)q^n=\dfrac{(-q;q)_{\infty}}{(q;q)_{\infty}}.\end{equation}
Again,  for any positive integer $\ell$, an $\ell$-regular partition of $n$ is a partition in which  no part is divisible by $\ell$. If $b_\ell(n)$ denotes the number of $\ell$-regular
partitions of $n$ (with $b_\ell(0)$ = 1), then the generating function of $b_\ell(n)$ is given by
\begin{equation}
	\sum_{n=0}^\infty{b_\ell(n)q^n}=\frac{f_\ell}{f_1}.
\end{equation}
Naika et. al.\cite{ms} defined a new overpartition functions known as $(j, k)-$regular overpartition. An overpartition of a positive integer $n$ is said to be  $(j, k)$-regular overpartition if none of the parts is congruent to $j \pmod k$. If $p_{j,k}(n)$ denotes the number of $(j, k)-$regular overpartition of $n$ (with $p_{j,k}(0)$ = 1), then its  generating function  is given by
\begin{equation}\label{p10}
 \sum_{n=0}^{\infty}\overline{p}_{j,k}(n)q^n=\dfrac{(-q;q)_{\infty}(q^j;q^k)_{\infty}}{(q;q)_{\infty}(-q^j;q^k)_{\infty}}.\end{equation}
 For example, the $(4,8)$-regular overpartition of 4 are given by 
$$\quad3+1,\quad\overline{3}+1,\quad 3+\overline{1},\quad\overline{3}+\overline{1},\quad2+2,\quad\overline{2}+2,\quad2+1+1,\quad\overline{2}+1+1,,\quad2+\overline{1}+1,$$$$\quad\overline{2}+\overline{1}+1,\quad1+1+1+1,\quad\overline{1}+1+1+1.
$$
Naika et. al.\cite{ms} obtain many infinite families of congruences modulo powers of 2 for $\overline{p}_{3,6}(n)$, $\overline{p}_{5,10}(n)$ and $\overline{p}_{9,18}(n)$. In this paper, we prove many infinite families of congruences modulo power of 2 for $\overline{p}_{4,8}(n)$, $\overline{p}_{6,12}(n)$ and $\overline{p}_{8,16}(n)$.

\section{Some $q$-Series Identities}
\begin{lemma} The following 2-dissections hold:

 \begin{equation}\label{1}
  \dfrac{1}{f_1^2}=\dfrac{f_8^5}{f_2^5f_{16}^2}+2q\dfrac{f_4^2f_{16}^2}{f_2^5f_8},
 \end{equation}
 
 \begin{equation}\label{2}
  \dfrac{1}{f_1^4}=\dfrac{f_4^{14}}{f_2^{14}f_{8}^4}+4q\dfrac{f_4^2f_{8}^4}{f_2^{10}},
 \end{equation}
 \begin{equation}\label{3}
 f_1^2=\dfrac{f_2f_8^5}{f_4^2f_{16}^2}-2q\dfrac{f_2f_{16}^2}{f_8},
 \end{equation}
 \begin{equation}\label{5}
      \dfrac{f_3^3}{f_1}=\dfrac{f_4^3f_6^2}{f_2^2f_{12}}+q\dfrac{f_{12}^3}{f_4}.
      \end{equation}
      \end{lemma}
  The identity \eqref{1} is the 2-dissection of $\phi(q)$ \cite[(1.9.4)]{pq}. The identity \eqref{2} is the
  2-dissection of $\phi(q)^2$ \cite[(1.10.1)]{pq}. The equations \eqref{3} can be obtained
  from the equations \eqref{1} by replacing $q$ by $-q$ respectively. The equation \eqref{5} is obtained from \cite[(22.1.14)]{pq}

 \begin{lemma}\cite{md} The following 3-dissections hold:
 \begin{equation}\label{4}
   \dfrac{f_2}{f_1^2}=\dfrac{f_6^4f_9^6}{f_3^8f_{18}^3}+2q\dfrac{f_6^3f_9^3}{f_3^7}+4q^2\dfrac{f_6^2f_{18}^3}{f_3^6},
   \end{equation}
   \begin{equation}\label{8}
      \dfrac{f_1^2}{f_2}=\dfrac{f_9^2}{f_{18}}-2q\dfrac{f_3f_{18}^2}{f_6f_9}.
      \end{equation}
  \end{lemma}
 \begin{lemma}\cite[Lemma 2.3]{lw} We have
 \begin{equation}\label{7}
       f_1^3= P(q^3)-3qf_9^3,
       \end{equation}
       where 
       $$P(q)= \sum_{m=-\infty}^{\infty}(-1)^m(6m+1)q^{m(3m+1)/2}=f(-q)\varphi(q)\varphi(q^3)+4qf(-q)\psi(q^2)\psi(q^6).$$
 \end{lemma}

 \begin{lemma} 
 \cite[Theorem 2.2]{cu} For any prime $p \geq 5$, we have
 $$f_1=\sum_{\substack{k=-(p-1)/2 \\ k\neq (\pm p-1)/6}
 }^{k=(p-1)/2}(-1)^kq^{(3k^2+k)/2}f\left( -q^{(3p^2+(6k+1)p)/2},-q^{(3p^2-(6k+1)p)/2}\right)$$

 \begin{equation} \label{pp}
 +(-1)^{(\pm p-1)/6}q^{(p^2-1)/24}f_{p^2}.
 \end{equation} where

 \begin{equation*}
 \dfrac{\pm p-1}{6}
 =\left\{
 \begin{array}{ll}
 \dfrac{ (p-1)}{6},
 & if~ p\equiv 1 \pmod{6},\\
 \dfrac{ (-p-1)}{6},
 & if~ p\equiv -1 \pmod{6}.
 \end{array}
 \right.
 \end{equation*}Furthermore, if\\
 $$\frac{-(p-1)}{2}\leq k \leq\frac{(p-1)}{2} \quad and \quad k \neq \frac{(\pm p-1)}{2},$$ then
 $$ \frac{(3k^2+k)}{2} \not\equiv  \frac{(p^2-1)}{24} \pmod{p}.$$
 \end{lemma}

  \begin{lemma} 
  From \cite{MD} We have that 
  \begin{equation} \label{g1}
  f_1=f_{25}(R(q^5)-q-q^2R(q^5)^{-1}),
  \end{equation} where $$R(q)=\dfrac{(q^2;q^5)_{\infty}(q^3;q^5)_{\infty}}{(q;q^5)_{\infty}(q^4;q^5)_{\infty}}.$$
  \end{lemma}

 \begin{lemma}\cite[p. 303, Entry 17(v)]{bc}  We have that
 \begin{equation}\label{u7}
 f_1=f_{49}\left(\dfrac{E(q^7)}{C(q^7)}-q \dfrac{D(q^7)}{E(q^7)}-q^2+q^5\dfrac{C(q^7)}{D(q^7)}\right),
 \end{equation}
 where $D(q)=f(-q^3,-q^4),  E(q)=f(-q^2,-q^5)~and ~C(q)=f(-q,-q^6)$.
 \end{lemma}

 From \cite[p.39, Entry 24(iii)]{bc}, we note that
 \begin{equation}\label{e3}
 f_1^3=\sum_{n=0}^\infty(-1)^n(2n+1)q^{n(n+1)/2}.
 \end{equation}

 From \eqref{e3}, we deduce the following identities:
 \begin{equation}\label{eq11}
  \hspace{-6.4cm}(i)\quad f_1^3= J_0(q^{7}) -3qJ_1(q^{7})+ 5q^3J_2(q^{7})-7q^6J_3(q^{7}),
  \end{equation} 
  where $J_0,J_1,J_2$, and $J_3$ are series involving integral powers of $q^{7}$. 
  $$\hspace{-3.6cm}(ii)\quad f_1^3= I_0(q^{11}) -3qI_1(q^{11})+ 5q^3I_2(q^{11})-7q^6I_3(q^{11})+9q^{10}I_4(q^{11})$$\begin{equation}\label{eq2}-11q^{15}I_5(q^{11}). \end{equation}
    where $I_0,I_1,I_2,I_3,I_4$ and $I_5$
     are series involving  integral powers of $q^{11}$.  
 $$\hspace{-3.0cm}(iii)\quad f_1^3= C_0(q^{13}) -3qC_1(q^{13})+ 5q^3C_2(q^{13})-7q^6C_3(q^{13})+9q^{10}C_4(q^{13})$$\begin{equation}\label{eq21}-11q^{15}C_5(q^{13}) +13 q^{21}C_6(q^{13}).
 \end{equation}
 where $C_0,C_1,C_2,C_3,C_4,C_5$ and $C_6$ are series involving  integral powers of $q^{13}$.
 $$\hspace{-2.5cm}(iv)\quad 
 f_1^3= D_0(q^{17}) -3 qD_1(q^{17})+ 5q^3D_2(q^{17}) -7q^6D_3(q^{17})+9q^{10}D_4(q^{17})$$\begin{equation}\label{eq31}-11q^{15}D_5(q^{17}) +13 q^{21}D_6(q^{17})-15q^{28}D_7(q^{17}) + 17q^{36}D_8(q^{17}). 
 \end{equation}
 where $D_0, D_1,D_2,D_3,D_4,D_5,D_6,D_7$ and $D_8$ are series involving integral powers of $q^{17}$. 
 $$\hspace{-.2cm}(v)\quad 
 f_1^3= F_0(q^{19}) -3 qF_1(q^{19})+ 5q^3F_2(q^{19}) -7q^6F_3(q^{19})+9q^{10}F_4(q^{19})-11q^{15}F_5(q^{19})$$\begin{equation}\label{eq4} +13 q^{21}F_6(q^{19})-15q^{28}F_7(q^{19}) + 17q^{36}F_8(q^{19}) - 19q^{45}F_9 (q^{19}).
 \end{equation}
 where $F_0,F_1,F_2,F_3,F_4,F_5,F_6,F_7,F_8$ and $F_9$ are series with integral powers of $q^{19}$. 
  
 In addition to above $q-$series identities, we will be using following congruence properties which follows from binomial theorem: For any positive integer $k$ and $m$,
 
 	\begin{equation}\label{fp1}
 		f_k^{2m}\equiv f_{2k}^m\pmod 2,\end{equation}
 	\begin{equation}\label{fp2}
 		f_k^{4m}\equiv f_{2k}^{2m}\pmod 4,\end{equation}
 	\begin{equation}\label{fp3}
 		f_k^{8m}\equiv f_{2k}^{4m}\pmod 8,\end{equation}
 	\begin{equation}\label{fp4}
 		f_k^{16m}\equiv f_{2k}^{8m}\pmod {16}.\end{equation}

 \section{Congruences for $\overline{p}_{4,8}(n)$}
 \begin{theorem}If $k\in \left\lbrace 1,2,3,4,5,6\right\rbrace$ and  $j\in  \left\lbrace 0,2,3,4 \right\rbrace$ . Then for all integers $n\geq 0$ and  $\alpha\geq 0$, we have
 \begin{align}\label{t11}
 &\sum_{n=0}^\infty{\overline{p}_{4,8}\left(  5^{2\alpha} \cdot7^{2\alpha}\left(  16n+6\right)  \right) q^n}\equiv 32 f_1f_8\pmod{64},\\
 \label{t12}
 &\sum_{n=0}^\infty{\overline{p}_{4,8}\left(  5^{2\alpha+1}\cdot7^{2\alpha}\left(16n+14\right)  \right) q^n}\equiv 32 qf_{5}f_{40}\pmod{64},\\
 \label{t13}
 &\overline{p}_{4,8}\left(  5^{2\alpha+1}\cdot7^{2\alpha}\left(  16(5n+j)+14\right)  \right) \equiv 0\pmod{64}.\\
 \label{t14}
  &\sum_{n=0}^\infty{\overline{p}_{4,8}\left(  5^{2\alpha}\cdot7^{2\alpha+1}\left(16n+10\right)  \right) q^n}\equiv 32 q^2f_{7}f_{56}\pmod{64},\\
  \label{t15}
   &\overline{p}_{4,8}\left(  5^{2\alpha}\cdot7^{2\alpha+1}\left(16(7n+k)+14\right)  \right) \equiv 0\pmod{64},
 \end{align}
 \end{theorem}
 \begin{proof}
 Setting $j=4$ and $k=8$ in \eqref{p10}, we obtain
 \begin{equation}\label{l1}
 \sum_{n=0}^{\infty}\overline{p}_{4,8}(n)q^n=\dfrac{(-q;q)_{\infty}(q^4;q^8)_{\infty}}{(q;q)_{\infty}(-q^4;q^8)_{\infty}}.\end{equation}
 Applying elementary $q$-operation and using \eqref{fff}, we obtain 
 \begin{equation}\label{p11}
 \sum_{n=0}^{\infty}\overline{p}_{4,8}(n)q^n=\dfrac{f_2f_4^2f_{16}}{f_1^2f_8^3}.
 \end{equation}
 Using \eqref{1} in \eqref{p11}, we obtain 
 \begin{equation}\label{p12}
 \sum_{n=0}^{\infty}\overline{p}_{4,8}(n)q^n= \dfrac{f_4^2f_8^2}{f_2^4f_{16}}+2q\dfrac{f_4^4f_{16}^3}{f_2^4f_8^4} .
 \end{equation}
Extracting the terms involving $q^{2n}$ and $q^{2n+1}$ from \eqref{p12}, we obtain
 \begin{equation}\label{p13}
 \sum_{n=0}^{\infty}\overline{p}_{4,8}(2n)q^n= \dfrac{f_2^2f_4^2}{f_1^4f_{8}}
 \end{equation}
 and
 \begin{equation}\label{p14}
 \sum_{n=0}^{\infty}\overline{p}_{4,8}(2n+1)q^n= 2\dfrac{f_2^4f_{8}^3}{f_1^4f_4^4} ,
 \end{equation} respectively. 
 Employing \eqref{2} in \eqref{p13}, we obtain 
 \begin{equation}\label{p131}
 \sum_{n=0}^{\infty}\overline{p}_{4,8}(2n)q^n=  \dfrac{f_4^{16}}{f_2^{12}f_{8}^5}+4q\dfrac{f_4^4f_{8}^3}{f_2^{8}}.
 \end{equation}
 Extracting the terms involving  $q^{2n}$ and $q^{2n+1}$ from  \eqref{p131}, we obtain
 \begin{equation}\label{p141}
 \sum_{n=0}^{\infty}\overline{p}_{4,8}(4n)q^n=  \dfrac{f_2^{16}}{f_1^{12}f_{4}^5}
 \end{equation}
 and
 \begin{equation}\label{p15}
 \sum_{n=0}^{\infty}\overline{p}_{4,8}(4n+2)q^n=  4\dfrac{f_2^4f_{4}^3}{f_1^{8}},
 \end{equation}respectively.
 Employing \eqref{2} in \eqref{p15}, we obtain
 \begin{equation}\label{p16}
 \sum_{n=0}^{\infty}\overline{p}_{4,8}(4n+2)q^n=  4 \dfrac{f_4^{31}}{f_2^{24}f_{8}^8}+32q\dfrac{f_4^{19}}{f_2^{20}} +64q^2\dfrac{f_4^7f_{8}^8}{f_2^{16}}.
 \end{equation}
 Extracting the terms involving $q^{2n+1}$ from \eqref{p16}, we obtain
 \begin{equation}\label{p17}
 \sum_{n=0}^{\infty}\overline{p}_{4,8}(8n+6)q^n= 32q\dfrac{f_2^{19}}{f_1^{20}} .
 \end{equation}
 Employing \eqref{fp1} in \eqref{p17}, we find that 
 \begin{equation}\label{p18}
 \sum_{n=0}^{\infty}\overline{p}_{4,8}(8n+6)q^n= 32f_2^{9}\pmod{64} .
 \end{equation}
 Extracting the terms involving $q^{2n}$ from \eqref{p18}, we obtain 
 \begin{equation}\label{p19}
 \sum_{n=0}^{\infty}\overline{p}_{4,8}(16n+6)q^n= 32f_1^{8}f_1\pmod{64}.
 \end{equation}
 Again, using \eqref{fp1} in \eqref{p19}, we obtain
 \begin{equation}\label{p20}
 \sum_{n=0}^{\infty}\overline{p}_{4,8}(16n+6)q^n= 32f_{8}f_1\pmod{64} .
 \end{equation}

    The equation \eqref{p20} is the case $\alpha=\beta=0$ of equation \eqref{t11}. Suppose that the congruence \eqref{t11} is true for any integer $\alpha\ge0$ with $\beta=0$. Utilising \eqref{g1} in \eqref{t11} with $\beta=0$ and then extracting the terms involving $q^{5n+4}$, we arrive at
     		  	  	    \begin{equation}\label{v1}
     		  	  	 \sum_{n=0}^\infty{\overline{p}_{4,8}\left(  5^{2\alpha+1}\left(  16n+14\right)  \right) q^n}\equiv 32 qf_5f_{40}\pmod{64}.
     		  	  	   \end{equation}
     		  	  	   Extracting the terms involving $q^{5n+1}$, from \eqref{v1}, we obtain
     		  	  	   \begin{equation}\label{v2}
     		  	  	   \sum_{n=0}^\infty{\overline{p}_{4,8}\left(  5^{2(\alpha+1)}\left(  16n+6\right)  \right) q^n}\equiv 32 f_1f_{8}\pmod{64},
     		  	  	   \end{equation}
     which implies that \eqref{t11} is true for $\alpha+1$ with $\beta=0$. By principle of mathematical induction, \eqref{t11} is true for all non negative integers $\alpha\ge0$ with $\beta=0$.
     		  	  	   Suppose that the congruence \eqref{t11} holds for $\alpha,\beta\ge0$. Utilising \eqref{u7} in \eqref{t11} and then extracting the terms involving $q^{7n+4}$, we obtain
     		  	  	   \begin{equation}\label{v3}
     		  	  	   \sum_{n=0}^\infty{\overline{p}_{4,8}\left(  5^{2\alpha}\cdot7^{2\alpha+1}\left(16n+10\right)  \right) q^n}\equiv 32 q^2f_{7}f_{56}\pmod{64},
     		  	  	   \end{equation} which proves \eqref{t14}.
     		  	  	   Now extracting the terms involving $q^{7n+2}$, from \eqref{v3}, we obtain
     		  	  	   \begin{equation}\label{v4}
     		  	  	  \sum_{n=0}^\infty{\overline{p}_{4,8}\left(  5^{2\alpha}\cdot7^{2(\alpha+1)}\left(16n+6\right)  \right) q^n}\equiv 32 f_{1}f_{8}\pmod{64},
     		  	  	   \end{equation}
     		  	  	  which implies that \eqref{t11} is true for all $\beta+1$. By principle of mathematical induction \eqref{t11} is true for all non-negative integers $\alpha, \beta$.

     		  	  	 Employing \eqref{g1} in \eqref{t11} and then extracting the terms involving $q^{5n+4}$, we arrive at \eqref{t12}.
     		  	  	  Again employing \eqref{g1} in \eqref{t12} and extracting the terms involving $q^{5n+j}$ for $j\in \{0,2,3,4\}$ from \eqref{t12}, we arrive at \eqref{t13}. Employing \eqref{u7} in \eqref{v3} and then extracting the terms involving $q^{7n+k}$ for $k\in \{0,1,3,4,5,6\}$, we arrive at \eqref{t15}.   
       \end{proof}
 
 \begin{theorem}\label{th4}
 Let $t\in{\{6,10,14\}}$, $s\in \{2,4\}$, $i\in \{1,2\}$, $u_1\in \{18,34\}$, $j \in \{1,2,3,4\}$ and $k \in \{1,2,3,4,5,6\}$ . Then for all integers $\alpha \geq0$, $\beta \geq0$, and  $\gamma \geq0$, we have
 \begin{align}\label{mf1a}
 &\overline{p}_{4,8}(16n+t)\equiv 0\pmod{32},\\
 \label{r1a}
 &\sum_{n=0}^{\infty}\overline{p}_{4,8}\Big(16\cdot3^{4\alpha}\cdot5^{4\beta}\cdot7^{4\gamma}(n)+2\cdot3^{4\alpha}\cdot5^{4\beta}\cdot7^{4\gamma}\Big)q^n\equiv 4f_1^3 \pmod{32},\\ 
 \label{s1}
 &\overline{p}_{4,8}\Big(16\cdot3^{4\alpha+1}\cdot5^{4\beta}\cdot7^{4\gamma}(n)+34\cdot3^{4\alpha}\cdot5^{4\beta}\cdot7^{4\gamma}\Big)\equiv 0 \pmod{32},\\
 \label{s2}
 &\sum_{n=0}^{\infty}\overline{p}_{4,8}\Big(16\cdot3^{4\alpha+1}\cdot5^{4\beta}\cdot7^{4\gamma}(n)+2\cdot3^{4\alpha+2}\cdot5^{4\beta}\cdot7^{4\gamma}\Big)q^n\equiv -12f_3^3 \pmod{32},\\
 	\label{s21}
 	 &\sum_{n=0}^{\infty}\overline{p}_{4,8}\Big(16\cdot3^{4\alpha+2}\cdot5^{4\beta}\cdot7^{4\gamma}(n)+2\cdot3^{4\alpha+2}\cdot5^{4\beta}\cdot7^{4\gamma}\Big)q^n\equiv -12f_1^3 \pmod{32},\\
 	 \label{s22}
  	 &\overline{p}_{4,8}\Big(16\cdot3^{4\alpha+2}\cdot5^{4\beta}\cdot7^{4\gamma}(3n+i)+2\cdot3^{4\alpha+2}\cdot5^{4\beta}\cdot7^{4\gamma}\Big)\equiv 0 \pmod{32},\\
 	\label{s3}
 		&\overline{p}_{4,8}\Big(16\cdot3^{4\alpha+3}\cdot5^{4\beta}\cdot7^{4\gamma}(n)+34\cdot3^{4\alpha+2}\cdot5^{4\beta}\cdot7^{4\gamma}\Big)\equiv 0 \pmod{32},\\
 	\label{s4}
  	&\sum_{n=0}^{\infty}\overline{p}_{4,8}\Big(16\cdot3^{4\alpha+3}\cdot5^{4\beta}\cdot7^{4\gamma}(n)+2\cdot3^{4(\alpha+1)}\cdot5^{4\beta}\cdot7^{4\gamma}\Big)q^n\equiv 4f_3^3 \pmod{32},\\
  	\label{s5}
  	 &\overline{p}_{4,8}\Big(16\cdot3^{4\alpha+4}\cdot5^{4\beta}\cdot7^{4\gamma}(n)+u_1\cdot3^{4(\alpha+1)}\cdot5^{4\beta}\cdot7^{4\gamma}\Big)\equiv 0 \pmod{32},\\
  	 \label{s6}
 	 &\sum_{n=0}^{\infty}\overline{p}_{4,8}\Big(16\cdot3^{4\alpha}\cdot5^{4\beta+1}\cdot7^{4\gamma}(n)+2\cdot3^{4\alpha}\cdot5^{4\beta+2}\cdot7^{4\gamma}\Big)q^n\equiv 20f_5^3 \pmod{32},\\
 	 \label{s66}
 	  	& \overline{p}_{4,8}\Big(16\cdot3^{4\alpha}\cdot5^{4\beta}\cdot7^{4\gamma}( 5n+s)+2\cdot3^{4\alpha}\cdot5^{4\beta}\cdot7^{4\gamma}\Big)\equiv 0 \pmod{32},\\
 	  \label{s8}
 	  	 &\sum_{n=0}^{\infty}\overline{p}_{4,8}\Big(16\cdot3^{4\alpha}\cdot5^{4\beta+2}\cdot7^{4\gamma}(n)+2\cdot3^{4\alpha}\cdot5^{4\beta+2}\cdot7^{4\gamma}\Big)q^n\equiv 20f_1^3 \pmod{32},\\
 \label{s7}
 	 &\overline{p}_{4,8}\Big(16\cdot3^{4\alpha}\cdot5^{4\beta+1}\cdot7^{4\gamma}(5n+j)+2\cdot3^{4\alpha}\cdot5^{4\beta+2}\cdot7^{4\gamma}\Big)q^n\equiv 0\pmod{32},\\
 	\label{s9}
 	 &\sum_{n=0}^{\infty}\overline{p}_{4,8}\Big(16\cdot3^{4\alpha}\cdot5^{4\beta+3}\cdot7^{4\gamma}(n)+2\cdot3^{4\alpha}\cdot5^{4(\beta+1)}\cdot7^{4\gamma}\Big)q^n\equiv 4f_5^3 \pmod{32},\\
 \label{s88}
  	 &\overline{p}_{4,8}\Big(16\cdot3^{4\alpha}\cdot5^{4\beta+2}\cdot7^{4\gamma}(5n+s)+2\cdot3^{4\alpha}\cdot5^{4\beta+2}\cdot7^{4\gamma}\Big)\equiv 0 \pmod{32},\\
  		\label{u1}
 	 &\overline{p}_{4,8}\Big(16\cdot3^{4\alpha}\cdot5^{4\beta+3}\cdot7^{4\gamma}(5n+j)+2\cdot3^{4\alpha}\cdot5^{4(\beta+1)}\cdot7^{4\gamma}\Big)\equiv 0 \pmod{32},\\
 	 \label{u2}
 &\sum_{n=0}^{\infty}\overline{p}_{4,8}\Big(16\cdot3^{4\alpha}\cdot5^{4\beta}\cdot7^{4\gamma+1}(n)+2\cdot3^{4\alpha}\cdot5^{4\beta}\cdot7^{4\gamma+2}\Big)q^n\equiv -4f_7^3 \pmod{32},\end{align}
 \begin{align}
 \label{u4}
&\sum_{n=0}^{\infty}\overline{p}_{4,8}\Big(16\cdot3^{4\alpha}\cdot5^{4\beta}\cdot7^{4\gamma+2)}(n)+3\cdot3^{4\alpha}\cdot5^{4\beta}\cdot7^{4\gamma+2}\Big)q^n\equiv -4f_1^3 \pmod{32},
\\\label{u3}
 &\overline{p}_{4,8}\Big(16\cdot3^{4\alpha}\cdot5^{4\beta}\cdot7^{4\gamma+1}(7n+k)+2\cdot3^{4\alpha}\cdot5^{4\beta}\cdot7^{4\gamma+2}\Big)\equiv 0 \pmod{32},\\
 \label{u5}
 &\sum_{n=0}^{\infty}\overline{p}_{4,8}\Big(16\cdot3^{4\alpha}\cdot5^{4\beta}\cdot7^{4\gamma+3}(n)+2\cdot3^{4\alpha}\cdot5^{4\beta}\cdot7^{4(\gamma+1)}\Big)q^n\equiv 4f_7^3 \pmod{32},\\
 \label{u6}
 &\overline{p}_{4,8}\Big(16\cdot3^{4\alpha}\cdot5^{4\beta}\cdot7^{4\gamma+3}(7n+k)+2\cdot3^{4\alpha}\cdot5^{4\beta}\cdot7^{4(\gamma+1)}\Big)\equiv 0 \pmod{32}.
 \end{align} 
  \end{theorem}
 
 \begin{proof} 
  From \eqref{p15}, we note that
 	  \begin{equation}\label{r1}
 	  \sum_{n=0}^{\infty} \overline{p}_{4,8}(4n+2)q^n\equiv\dfrac{4f_2^4f_4^3}{f_1^8}\pmod{32},
 	  \end{equation}
 	  Employing \eqref{fp3} in \eqref{r1}, we obtain 	  
 	  \begin{equation}\label{r2}
 	  \sum_{n=0}^{\infty} \overline{p}_{4,8}(4n+2)q^n\equiv4f_4^3\pmod{32}.
 	  \end{equation}
Extracting  the terms involving $q^{4n+1}$, $q^{4n+2}$ and $q^{4n+3}$  from \eqref{r2}, we arrive at \eqref{mf1a}. Again, extracting the terms involving $q^{4n}$ from \eqref{r2}, we obtain
 	  \begin{equation}\label{r3}
 	  \sum_{n=0}^{\infty}\overline{p}_{4,8}(16n+2)q^n\equiv4f_1^3\pmod{32}.
 	  \end{equation}
 	  The equation \eqref{r3} is the case $\alpha=\beta= \gamma=0$ of equation \eqref{r1a}. Suppose that the congruence \eqref{r1a} is true for any integer $\alpha\ge0$ with $\beta=\gamma=0$. Utilising \eqref{7} in \eqref{r1a} with $\beta=\gamma=0$ and then extracting the terms involving $q^{3n+1}$, we arrive at
  \begin{equation}\label{leq181}
  \sum_{n=0}^{\infty}\overline{p}_{4,8}\Big(16\cdot3^{4\alpha+1}(n)+2\cdot3^{4\alpha+2}\Big)q^n\equiv -12f_3^3\pmod{32}.
 \end{equation}
 Extracting the terms involving $q^{3n}$, from \eqref{leq181}, we obtain
 \begin{equation}\label{leq191}
 \sum_{n=0}^{\infty}\overline{p}_{4,8}\Big(16\cdot3^{4\alpha+2}(n)+2\cdot3^{4\alpha+2)}\Big)q^n\equiv -12 f_1^3\pmod{32}.
 \end{equation}
  Utilising \eqref{7} in \eqref{leq191} and then extracting the terms involving $q^{3n+1}$, we arrive at
  \begin{equation}\label{leq18}
  \sum_{n=0}^{\infty}\overline{p}_{4,8}\Big(16\cdot3^{4\alpha+3}(n)+2\cdot3^{4(\alpha+1)}\Big)q^n\equiv 4f_3^3\pmod{32}.
 \end{equation}
 Extracting the terms involving $q^{3n}$, from \eqref{leq18}, we obtain
 \begin{equation}\label{leq19}
 \sum_{n=0}^{\infty}\overline{p}_{4,8}\Big(16\cdot3^{4(\alpha+1)}(n)+2\cdot3^{4(\alpha+1)}\Big)q^n\equiv 4 f_1^3\pmod{32},
 \end{equation}
 which implies that \eqref{r1a} is true for $\alpha+1$ with $\beta=\gamma=0$. By principle of mathematical induction, \eqref{r1a} is true for all $\alpha$.
 Suppose that the congruence \eqref{r1a} holds for $\alpha,\beta\ge0$ with $\gamma=0$. Utilising \eqref{g1} in \eqref{r1a} and then extracting the terms involving $q^{5n+3}$, we obtain
 \begin{equation}\label{le201}
 \sum_{n=0}^{\infty}\overline{p}_{4,8}\Big(16\cdot3^{4\alpha}\cdot5^{4\beta +1}(n)+2\cdot3^{4\alpha}\cdot5^{4\beta+2}\Big)q^n\equiv 20f_5^3 \pmod{32}.
 \end{equation}
 Extracting the terms involving $q^{5n}$ from \eqref{le201}, we obtain
 \begin{equation}\label{le211}
 \sum_{n=0}^{\infty}\overline{p}_{4,8}\Big(16\cdot3^{4\alpha}\cdot5^{4\beta +2}(n)+2\cdot3^{4\alpha}\cdot 5^{4\beta+2}\Big)q^n\equiv 20f_1^3 \pmod{32}.
 \end{equation}
 Utilising \eqref{g1} in \eqref{le211} and then extracting the terms involving $q^{5n+3}$, we obtain
 \begin{equation}\label{le20}
 \sum_{n=0}^{\infty}\overline{p}_{4,8}\Big(16\cdot3^{4\alpha}\cdot5^{4\beta +3}(n)+2\cdot3^{4\alpha}\cdot5^{4(\beta+1)}\Big)q^n\equiv 4f_5^3 \pmod{32}.
 \end{equation}
 Extracting the terms involving $q^{5n}$, from \eqref{le20}, we obtain
 \begin{equation}\label{le21}
 \sum_{n=0}^{\infty}\overline{p}_{4,8}\Big(16\cdot3^{4\alpha}\cdot5^{4(\beta +1)}(n)+2\cdot3^{4\alpha}\cdot 5^{4(\beta+1)}\Big)q^n\equiv 4f_1^3 \pmod{32},
 \end{equation}
 which implies that \eqref{r1a} is true for $\beta+1$ with $\gamma=0$. By principle of mathematical induction, \eqref{r1a} is true for all non-negative integers $\alpha, \beta$ with $\gamma=0$. Suppose that the congruence \eqref{r1a} holds for $\alpha,\beta,\gamma\ge0$. Utilising \eqref{u7} in \eqref{r1a} and then extracting the terms involving $q^{7n+6}$, we obtain
 \begin{equation}\label{le221}
 \sum_{n=0}^{\infty}\overline{p}_{4,8}\Big(16\cdot3^{4\alpha}\cdot5^{4\beta}\cdot7^{4\gamma+1}(n)+2\cdot3^{4\alpha}\cdot5^{4\beta}\cdot7^{4\gamma+2}\Big)q^n\equiv -4f_7^3 \pmod{32},
 \end{equation} which proves \eqref{u2}. Extracting the terms involving $q^{7n}$ from \eqref{le221}, we obtain 
 \begin{equation}\label{le222}
 \sum_{n=0}^{\infty}\overline{p}_{4,8}\Big(16\cdot3^{4\alpha}\cdot5^{4\beta}\cdot7^{4\gamma+2}(n)+2\cdot3^{4\alpha}\cdot5^{4\beta}\cdot7^{4\gamma+2}\Big)q^n\equiv -4f_1^3 \pmod{32}, 
 \end{equation} 
 which proves \eqref{u4}. Utilising \eqref{u7} in \eqref{le222} and then extracting the terms involving $q^{7n+6}$, we obtain 
 \begin{equation}\label{le21k}
 \sum_{n=0}^{\infty}\overline{p}_{4,8}\Big(16\cdot3^{4\alpha}\cdot5^{4\beta}\cdot7^{4\gamma+3}(n)+2\cdot3^{4\alpha}\cdot5^{4\beta}\cdot7^{4(\gamma+1)}\Big)q^n\equiv 4f_7^3 \pmod{32},
 \end{equation} which proves \eqref{u5}. Extracting the terms involving $q^{7n}$ from \eqref{le22}, we obtain 
 \begin{equation}\label{le22}
 \sum_{n=0}^{\infty}\overline{p}_{4,8}\Big(16\cdot3^{4\alpha}\cdot5^{4\beta}\cdot7^{4(\gamma+1)}(n)+2\cdot3^{4\alpha}\cdot5^{4\beta}\cdot7^{4(\gamma+1)}\Big)q^n\equiv 4f_1^3 \pmod{32}, 
 \end{equation} 
 which implies that \eqref{r1a} is true for all $\gamma+1$. By mathematical induction \eqref{r1a} is true for all non-negative integers $\alpha, \beta, \gamma$. 
 
 Employing \eqref{7} in \eqref{r1a} and then extracting the terms involving $q^{3n+2}$, we arrive at \eqref{s1}.
 Again, employing \eqref{7} in \eqref{r1a} and extracting the terms involving $q^{3n+1}$ from both sides, we arrive at \eqref{s2}. Extracting the terms involving powers of $q^{3n}$ and  $q^{3n+i}$ for $i\in \{1,2\}$ from \eqref{s2}, we arrive at  \eqref{s21} and \eqref{s22}, respectively. 
 
 Employing \eqref{7} in \eqref{s21} and then extracting the terms involving $q^{3n+2}$ and $q^{3n+1}$ from both sides, we arrive at \eqref{s3} and \eqref{s4}, respectively.  
 Again, extracting the terms involving $q^{3n+i}$ for $i\in \{1,2\}$ from \eqref{s4}, we arrive at \eqref{s5}.   
 
Employing \eqref{g1} in \eqref{r1a}, then extracting the terms involving $q^{5n+3}$ and $q^{5n+s}$ for $s\in\{2,4\}$, we arrive at \eqref{s6} and \eqref{s66}, respectively. Extracting the terms involving $q^{5n}$ and $q^{5n+j}$ for $j\in\{1,2,3,4\}$ from \eqref{s6}, we arrive at \eqref{s7} and \eqref{s8}, respectively. 
 Again, employing \eqref{g1} in \eqref{s8} and then extracting the terms involving $q^{5n+3}$ and $q^{5n+s}$ for $s\in\{2,4\}$ from \eqref{s8}, we arrive at \eqref{s9} and \eqref{s88}, respectively. Extracting the terms involving $q^{5n+j}$ for $j=\{1,2,3,4\}$ from \eqref{s9}, we arrive at \eqref{u1}.  
 
  Employing \eqref{u7} in \eqref{r1a} and then extracting the terms involving $q^{7n+6}$, we arrive at \eqref{u2}. Extracting the terms involving $q^{7n+k}$ for $k\in\{1,2,3,4,5,6\}$ from \eqref{u2}, we arrive at \eqref{u3}. 
  Finally, extracting the terms involving $q^{7n+k}$ for $k\in\{1,2,3,4,5,6\}$ from \eqref{u5}, we complete the proof of \eqref{u6}. 
   \end{proof}

  \begin{theorem}For all integers $n\geq 0$ and $\alpha\geq 0$, we have  
   \begin{align}\label{t31}
       &\overline{p}_{4,8}(16(7n+j)+2) \equiv 0\pmod{32},\quad for \quad j\in 2,4,5,6,\\
      \label{t32}
       &\overline{p}_{4,8}(16(11n+k)+2) \equiv 0\pmod{32},\quad for \quad k\in 2,4,5,7,8,9,\\
    \label{t33}
        &\overline{p}_{4,8}(16(13n+m)+2) \equiv 0\pmod{32},\quad for \quad m\in 4,5,7,6,9,11,12,\end{align}
    \begin{align}
   \label{t34}
          &\overline{p}_{4,8}(16(17n+w)n+2) \equiv 0\pmod{32},\quad for \quad w\in 2, 5, 7, 8, 9, 12, 13, 14, 16,\\
       \label{t35}
             &\overline{p}_{4,8}(16(19n+t)n+2) \equiv 0\pmod{32},\quad for \quad s\in4, 5, 7,  8, 11, 12, 13, 14, 16, 18.
            \end{align}
  \end{theorem}
  \begin{proof}
  From \eqref{r3}, we have
   \begin{equation}\label{rp1}
   \sum_{n=0}^{\infty}\overline{p}_{4,8}(16n+2)q^n\equiv4f_1^3\pmod{32}.
   \end{equation}
   Employing \eqref{eq11} in \eqref{rp1}, we obtain
   \begin{equation}\label{rp2}
      \sum_{n=0}^{\infty}\overline{p}_{4,8}(16n+2)q^n\equiv4\left( J_0(q^{7}) -3qJ_1(q^{7})+ 5q^3J_2(q^{7})-7q^6J_3(q^{7})\right) \pmod{32}.
      \end{equation}
  Extracting the terms involving $q^{7n+j}$ from \eqref{rp2}, we obtain
    \begin{equation}\label{rp3}
      \overline{p}_{4,8}(16(7n+j)+2) \equiv 0\pmod{32},
      \end{equation}
      which proves \eqref{t31}.       
      Proofs of \eqref{t32} -\eqref{t35} are identitical to the proof of \eqref{t31} and  we employ  \eqref{eq2}-\eqref{eq4}, respectively. 
  \end{proof}

  	\begin{theorem} Let $j \in \{0,2,3,4\}$ and $k \in \{0,1,3,4,5,6\}$ . Then for all integers $\alpha \geq0$ and $\beta \geq0$, we have  	
  	\begin{equation}\label{y1}
  	   \sum_{n=0}^{\infty}\overline{p}_{4,8}\Big(8\cdot5^{2\alpha}\cdot7^{2\beta}(n)+3\cdot5^{2\alpha}\cdot7^{2\beta}\Big)q^n\equiv8f_1^9 \pmod{16}, 
  	   \end{equation}  	  
  	  \begin{equation}\label{y2}
  	  \sum_{n=0}^{\infty}\overline{p}_{4,8}\Big(8\cdot5^{2\alpha+1}\cdot7^{2\beta}(n)+7\cdot5^{2\alpha+1}\cdot7^{2\beta}\Big)q^n\equiv8qf_5^9 \pmod{16}, 
  	    \end{equation}    	
  	  \begin{equation}\label{y3}
  	      \overline{p}_{4,8}\Big(8\cdot5^{2\alpha+1}\cdot7^{2\beta}(5n+j)+7\cdot5^{2\alpha+1}\cdot7^{2\beta}\Big)\equiv0 \pmod{16}, 
  	     \end{equation}   	     
  	 \begin{equation}\label{y4}
  	       \sum_{n=0}^{\infty}\overline{p}_{4,8}\Big(8\cdot5^{2\alpha}\cdot7^{2\beta+1}(n)+5\cdot5^{2\alpha}\cdot7^{2\beta+1}\Big)q^n\equiv8q^2f_7^9 \pmod{16}, 
  	     \end{equation}     
  	    	  \begin{equation}\label{y5}
  	        \overline{p}_{4,8}\Big(8\cdot5^{2\alpha}\cdot7^{2\beta+1}(7n+k)+5\cdot5^{2\alpha}\cdot7^{2\beta+1}\Big)\equiv 0 \pmod{16}. 
  	      \end{equation}  	
  	\end{theorem}
  	\begin{proof}
  		 From\eqref{p14}, we have
  		 \begin{equation}\label{q1}
  		 \sum_{n=0}^{\infty}\overline{p}_{4,8}(2n+1)q^n= 2\dfrac{f_2^4f_{8}^3}{f_1^4f_4^4}.
  		 \end{equation}
  		 Using \eqref{2} in \eqref{q1}, we obtain
  		 \begin{equation}\label{q3}
  		 \sum_{n=0}^{\infty}\overline{p}_{4,8}(2n+1)q^n= 2 \dfrac{f_4^{10}}{f_2^{10}f_{8}}+8q\dfrac{f_{8}^7}{f_2^{6}f_4^2}.
  		 \end{equation}
  		Now extracting the terms involving the powers of $q^{2n+1}$ from both sides of \eqref{q3}, we obtain
  		 \begin{equation}\label{q5}
  		 \sum_{n=0}^{\infty}\overline{p}_{4,8}(4n+3)q^n= 8\dfrac{f_{4}^7}{f_1^{6}f_2^2}.
  		 \end{equation}
  		 Employing \eqref{fp1} in \eqref{q5}, we obtain  		 
  \begin{equation}\label{q7}
   \sum_{n=0}^{\infty}\overline{p}_{4,8}(4n+3)q^n\equiv 8{f_{2}^9\pmod{16}}. 
  \end{equation}
   Extracting the terms involving $q^{2n}$ from \eqref{q7}, we obtain
  		 \begin{equation}\label{q13}
  		 \sum_{n=0}^{\infty}\overline{p}_{4,8}(8n+3)q^n\equiv 8{f_{1}^9\pmod{16}}.
  		 \end{equation}
  		 The equation \eqref{q13} is the case $\alpha=\beta=0$ of equation \eqref{y1}. Suppose that the congruence \eqref{y1} is true for any integer $\alpha\ge0$ with $\beta=0$. Utilising \eqref{g1} in \eqref{y1} with $\beta=0$ and then extracting the terms involving $q^{5n+4}$, we arrive at
  		  	  	    \begin{equation}\label{c1}
  		  	  	    \sum_{n=0}^{\infty}\overline{p}_{4,8}\Big(8\cdot5^{2\alpha+1}(n)+7\cdot5^{2\alpha+1}\Big)q^n\equiv8qf_5^9 \pmod{16}.
  		  	  	   \end{equation}
  		  	  	   Extracting the terms involving $q^{5n+1}$, from \eqref{c1}, we obtain
  		  	  	   \begin{equation}\label{c2}
  		  	  	   \sum_{n=0}^{\infty}\overline{p}_{4,8}\Big(8\cdot5^{2(\alpha+1)}(n)+3\cdot5^{2(\alpha+1)}\Big)q^n\equiv8f_1^9 \pmod{16},
  		  	  	   \end{equation}
  which implies that \eqref{y1} is true for $\alpha+1$ with $\beta=0$. By principle of mathematical induction, \eqref{y1} is true for all non negative integers $\alpha\ge0$ with $\beta=0$.
  		  	  	   Suppose that the congruence \eqref{y1} holds for $\alpha,\beta\ge0$. Utilising \eqref{u7} in \eqref{y1} and then extracting the terms involving $q^{7n+4}$, we obtain
  		  	  	   \begin{equation}\label{c3}
  		  	  	    \sum_{n=0}^{\infty}\overline{p}_{4,8}\Big(8\cdot5^{2\alpha}\cdot7^{2\beta+1}(n)+5\cdot5^{2\alpha}\cdot7^{2\beta+1}\Big)q^n\equiv8q^2f_7^9 \pmod{16},
  		  	  	   \end{equation} which proves \eqref{y4}.
  		  	  	   Now extracting the terms involving $q^{7n+2}$, from \eqref{c3}, we obtain
  		  	  	   \begin{equation}\label{c4}
  		  	  	  \sum_{n=0}^{\infty}\overline{p}_{4,8}\Big(8\cdot5^{2\alpha}\cdot7^{2(\beta+1)}(n)+3\cdot5^{2\alpha}\cdot7^{2(\beta+1)}\Big)q^n\equiv8f_1^9 \pmod{16},
  		  	  	   \end{equation}
  		  	  	  which implies that \eqref{y1} is true for all $\beta+1$. By principle of mathematical induction \eqref{y1} is true for all non-negative integers $\alpha, \beta$.

  		  	  	 Employing \eqref{g1} in \eqref{y1} and then extracting the terms involving $q^{5n+4}$, we arrive at \eqref{y2}.
  		  	  	  Again employing \eqref{g1} in \eqref{y2} and extracting the terms involving $q^{5n+j}$ for $j\in \{0,2,3,4\}$ from \eqref{y2}, we arrive at \eqref{y3}. Employing \eqref{u7} in \eqref{c3} and then extracting the terms involving $q^{7n+k}$ for $k\in \{0,1,3,4,5,6\}$, we arrive at \eqref{y5}.
  	\end{proof}
  	
   \begin{theorem}If $1\le j \le (p-1)$, then for all integers $n\geq 0$, $\alpha\geq 0$ and , we have
  
            \begin{equation}\label{t71}
            \sum_{n=0}^{\infty}\overline{p}_{4,8}\left( 2 \cdot p^{2\alpha}(8n+1)\right)  q^n \equiv f_1f_{2} \pmod{8},
            \end{equation}
            \begin{equation}\label{t72}
                \overline{p}_{4,8}\left( 2 \cdot p^{2(\alpha+1)}(8(pn+j)+1)\right) \equiv 0 \pmod{8}.
                \end{equation}
             \end{theorem}             
       \begin{proof}
    From \eqref{p15}, we have
     \begin{equation}\label{411}
     \sum_{n=0}^{\infty}\overline{p}_{4,8}(4n+2)q^n=  4\dfrac{f_2^4f_{4}^3}{f_1^{8}},
     \end{equation}
     Using \eqref{fp3} in \eqref{312}, we obtain
     \begin{equation}\label{412}
     \sum_{n=0}^{\infty}\overline{p}_{4,8}(4n+2)q^n\equiv   4f_4^{3}\pmod{8}.
     \end{equation}
     Extracting the terms involving $q^{4n}$ from \eqref{412}, we obtain
     \begin{equation}\label{414}
     \sum_{n=0}^{\infty}\overline{p}_{4,8}(16n+2)q^n\equiv   4f_1f_2\pmod{8}.
     \end{equation} Congruence \eqref{414} is the $\alpha = 0$ case of \eqref{t71}. Suppose that congruence \eqref{t71} is true for all $\alpha\geq 0$. Utilising \eqref{pp} in \eqref{t71}, we obtain 
    \begin{equation*}
    \sum_{n=0}^{\infty}\overline{p}_{4,8}\left( 2 \cdot p^{2\alpha}(8n+1)\right)  q^n \equiv \Big\{ \sum_{\substack{k=-(p-1)/2 \\ k\neq (\pm p-1)/6}
    }^{k=(p-1)/2}(-1)^kq^{(3k^2+k)/2}f\left( -q^{(3p^2+(6k+1)p)/2},-q^{(3p^2-(6k+1)p)/2}\right)
    \end{equation*}
    \begin{equation*} 
    +(-1)^{(\pm p-1)/6}q^{(p^2-1)/24}f_{p^2} \Big\} \quad 
    \end{equation*}
    $$\times \Big\{ \sum_{\substack{m=-(p-1)/2 \\ m\neq (\pm p-1)/6}
    }^{m=(p-1)/2}(-1)^mq^{(3m^2+m)}f\left( -q^{(3p^2+(6m+1)p)},-q^{(3p^2-(6m+1)p)}\right) $$
    \begin{equation}\label{f1f51}
    +(-1)^{(\pm p-1)/6}q^{(p^2-1)/12}f_{2p^2} \Big\}\pmod{8}.
    \end{equation}    
    Consider the congruence 
    $$\frac{(3k^2+k)}{2}+(3m^2+m) \equiv \frac{(p^2-1)}{8}\pmod{p},$$
    which is equal to
    $$(6k+1)^2+2(6m+1)^2 \equiv 0 \pmod{p}.$$
    For $\left(\dfrac{-2}{p}\right)  = -1$, the above congruence has only solution $k = m = \left( \dfrac{\pm p-1}{6}\right) $. Therefore,
    extracting the terms involving $q^{{pn}+(p^2-1)/8}$ from \eqref{f1f51}, dividing throughout
    by $q^{(p^2-1)/8}$ and then replacing $q^p$ by $q$, we obtain    
    \begin{equation}\label{s7f}
    \sum_{n=0}^{\infty}\overline{p}_{4,8}\left( 2 \cdot p^{2\alpha+1}(8n+p)\right)  q^n \equiv f_pf_{2p} \pmod{8}. 
    \end{equation}
    Extracting the terms involving $q^{pn}$ from \eqref{s7f} and replacing $q^p$ by $q$, we obtain
    \begin{equation}\label{s8f}
     \sum_{n=0}^{\infty}\overline{p}_{4,8}\left( 2 \cdot p^{2(\alpha+1)}(8n+1)\right)  q^n \equiv f_1f_{2} \pmod{8}, \end{equation}
    which is the $\alpha$ + 1 case of \eqref{t71}. Thus, by the principle of mathematical induction, we
    arrive at \eqref{t71}.
   Comparing the terms involving $q^{pn+j}$ for $ 1\leq j \leq p - 1$, from of \eqref{s7f}, we complete the proof of \eqref{t72}.
    \end{proof}
   
  \begin{theorem}For all integers $n\geq 0$, $\alpha\geq 0$  and $k \in \{1,2\}$, we have 
  \begin{equation}\label{t41}
  \sum_{n=0}^\infty{\overline{p}_{4,8}\left(   8\cdot3^{2\alpha}n  \right) q^n}\equiv \dfrac{f_1^2}{f_2}\pmod{4},
  \end{equation}
  \begin{equation}\label{t42}
      \sum_{n=0}^\infty{\overline{p}_{4,8}\left(   8\cdot3^{2\alpha+1}n \right) q^n}\equiv \dfrac{f_3^2}{f_6}\pmod{4},
      \end{equation}
  \begin{equation}\label{t43}
    {\overline{p}_{4,8}\left(   8\cdot3^{2\alpha+1}n+16\cdot3^{2\alpha} \right)}\equiv 0\pmod{4},
    \end{equation}
  \begin{equation}\label{t44}
 {\overline{p}_{4,8}\left(   8\cdot3^{2\alpha}(3n+k)  \right) }\equiv 0\pmod{4}.
  \end{equation}
   \end{theorem}   
  \begin{proof}
   From \eqref{p14}, we have
   \begin{equation}\label{311}
   \sum_{n=0}^{\infty}\overline{p}_{4,8}(4n)q^n=  \dfrac{f_2^{16}}{f_1^{12}f_{4}^5}.
   \end{equation}
   Using \eqref{fp2} in \eqref{312}, we obtain  
   \begin{equation}\label{312}
   \sum_{n=0}^{\infty}\overline{p}_{4,8}(4n)q^n=  \dfrac{f_2^2}{f_{4}}\pmod{4}.
   \end{equation}
   Extracting the terms involving $q^{2n}$ from \eqref{312}, we obtain
   \begin{equation}\label{313}
   \sum_{n=0}^{\infty}\overline{p}_{4,8}(8n)q^n=  \dfrac{f_1^2}{f_{2}}\pmod{4}.  
   \end{equation}
    Congruence \eqref{313} is the $\alpha = 0$ case of \eqref{t41}. Assume that \eqref{t41} is true for all $\alpha \geq 0$.
     Using  \eqref{8} in \eqref{t41}. Then extracting the terms involving powers of $q^{3n}$ from both sides, dividing throughout by $q^3$ and replacing $q^3$ by $q$, we obtain     
     \begin{equation}\label{314}
       \sum_{n=0}^\infty{\overline{p}_{4,8}\left(   8\cdot3^{2\alpha+1}n \right) q^n}\equiv \dfrac{f_3^2}{f_6}\pmod{4},
      \end{equation}
     which proves \eqref{t42}.  Again extracting the terms involving $q^{3n}$ from both sides and replacing $q^3$ by $q$, we obtain
    \begin{equation}\label{315}
          \sum_{n=0}^\infty{\overline{p}_{4,8}\left(   8\cdot3^{2(\alpha+1)}n \right) q^n}\equiv \dfrac{f_1^2}{f_2}\pmod{4},
         \end{equation}
          which is the $\alpha$ + 1 case of \eqref{t41}. Hence, by the principle of mathematical induction, we
                 arrive at \eqref{t41}. Now using \eqref{8} in \eqref{t41} then extracting the terms involving $q^{3n+2}$ from both sides, dividing throughout by $q^2$ and replacing $q^3$ by $q$ we prove \eqref{t43}.
                 Again, extracting the terms involving $q^{3n+k}$ for $k\in\left\lbrace 1,2\right\rbrace $ from both sides \eqref{314} and replacing $q^3$ by $q$ we prove \eqref{t44}.  
  \end{proof}
  \begin{theorem}For all integers $n\geq 0$ and $\alpha\geq 0$, we have
    \begin{equation}\label{t51}
    \sum_{n=0}^\infty{\overline{p}_{4,8}(24\cdot2^{2\alpha}n+8\cdot2^{2\alpha} )q^n\equiv 2\dfrac{f_{3}^3}{f_1}}\pmod{4},
    \end{equation}
    \begin{equation}\label{t52}
        \sum_{n=0}^\infty{\overline{p}_{4,8}(24\cdot2^{2\alpha+1}n+8\cdot2^{2(\alpha+1)})q^n\equiv 2\dfrac{f_{6}^3}{f_2}}\pmod{4},
        \end{equation}
    \begin{equation}\label{t53}
     {\overline{p}_{4,8}(24\cdot2^{2(\alpha+1)}n+20\cdot2^{2(\alpha+1)}))q^n}\equiv 0\pmod{4}.
      \end{equation}   
     \end{theorem}
    \begin{proof}
  Employing \eqref{8} in \eqref{313}, we obtain
 \begin{equation}\label{3a1}
    \sum_{n=0}^{\infty}\overline{p}_{4,8}(8n)q^n=  \left(  \dfrac{f_9^2}{f_{18}}+2q\dfrac{f_3f_{18}^2}{f_6f_9}\right) \pmod{4}. 
 \end{equation}
  Extracting the terms involving $q^{3n+1}$ from \eqref{3a1}, we obtain
  \begin{equation}\label{3a2}
   \sum_{n=0}^{\infty}\overline{p}_{4,8}(24n+8)q^n=   2\dfrac{f_1f_{6}^2}{f_2f_3} \pmod{4}.
   \end{equation}
  Using \eqref{fp1} in \eqref{3a2}, we obtain
 
  \begin{equation}\label{3a4}
   \sum_{n=0}^{\infty}\overline{p}_{4,8}(24n+8)q^n=   2\dfrac{f_{3}^3}{f_1} \pmod{4}.
   \end{equation}
    Congruence \eqref{3a4} is the $\alpha = 0$ case of \eqref{t51}. Assume that \eqref{t51} is true for all $\alpha \geq 0$. Using  \eqref{5} in \eqref{t51}, we obtain   
  \begin{equation}\label{3a5}
     \sum_{n=0}^\infty{\overline{p}_{4,8}(24\cdot2^{2\alpha}n+8\cdot2^{2\alpha} )q^n\equiv 2\left(\dfrac{f_4^3f_6^2}{f_2^2f_{12}}+q\dfrac{f_{12}^3}{f_4} \right)}\pmod{4}.
     \end{equation}  
  Then extracting the terms involving $q^{2n+1}$ from both sides, dividing throughout by $q$ and replacing $q^2$ by $q$, we obtain
   \begin{equation}\label{3a6}
       \sum_{n=0}^\infty{\overline{p}_{4,8}(24\cdot2^{2\alpha+1}n+8\cdot2^{2\alpha+2} )q^n\equiv 2\dfrac{f_{6}^3}{f_2}}\pmod{4},
       \end{equation} 
  which proves \eqref{t52}.
 Again extracting the terms involving $q^{2n}$ from both sides, dividing throughout by $q$ and replacing $q^2$ by $q$, we obtain  
 \begin{equation}\label{3a7}
        \sum_{n=0}^\infty{\overline{p}_{4,8}(24\cdot2^{2(\alpha+1)}n+8\cdot2^{2(\alpha+1)} )q^n\equiv 2\dfrac{f_{3}^3}{f_1}}\pmod{4},
        \end{equation} 
  which is the $\alpha$ + 1 case of \eqref{t51}. Hence, by the principle of mathematical induction, we
                  arrive at \eqref{t51}. Then extracting the terms involving $q^{2n+1}$ from \eqref{3a6}, dividing throughout by $q$ and replacing $q^2$ by $q$, we prove \eqref{t53}.                  
  \end{proof}
  
  \begin{theorem} If $j\in \{1,2,3,4,5,6,7\}$, then for all integers $n\geq 0$ and $\alpha\geq 0$, we have     
      \begin{equation}\label{t61}
     \overline{p}_{4,8}\left( 48(8n+j)+8\right) q^n= 0\pmod{4},
       \end{equation}   
      \begin{equation}\label{t62}
          \sum_{n=0}^{\infty}\overline{p}_{4,8}\left( 384\cdot p^{2\alpha} n+8(2p^{2\alpha}-1)\right)  q^n \equiv f_1 \pmod{4},
          \end{equation}
      \begin{equation}\label{t63}
        \overline{p}_{4,8}\left( 384\cdot p^{2\alpha+1} (pn+j)+8(2p^{2\alpha}-1)\right)  q^n \equiv 0 \pmod{4}.
        \end{equation}   
       \end{theorem}
   \begin{proof}
 Employing \eqref{5} in \eqref{3a4}, we obtain
 \begin{equation}\label{3a8}
    \sum_{n=0}^{\infty}\overline{p}_{4,8}(24n+8)q^n=   2\left(\dfrac{f_4^3f_6^2}{f_2^2f_{12}}+q\dfrac{f_{12}^3}{f_4} \right) \pmod{4}.
    \end{equation}
Extracting the terms involving $q^{2n}$ from \eqref{3a8}, we obtain
\begin{equation}\label{3a9}
   \sum_{n=0}^{\infty}\overline{p}_{4,8}(48n+8)q^n= 2\dfrac{f_4^3f_6^2}{f_2^2f_{12}}\pmod{4}.
\end{equation}
Employing \eqref{fp1} in \eqref{3a9}, we obtain
\begin{equation}\label{3b2}
         \sum_{n=0}^{\infty}\overline{p}_{4,8}(48n+8)q^n= 2f_8\pmod{4}.
         \end{equation}
   Extracting the terms involving $q^{8n+j}$ for $j\in \{1,2,3,4,5,6,7\}$ from \eqref{3b2}, we arrive at \eqref{t61}.       
 Again extracting the terms involving $q^{8n}$ from \eqref{3b2}, we obtain 
        
  \begin{equation}\label{3b3}
  \sum_{n=0}^{\infty}\overline{p}_{4,8}(384n+8)q^n= 2f_1\pmod{4},
  \end{equation} 
  which is the $\alpha = 0$ case of \eqref{t62}. Assume \eqref{t62} is true for any 
        $\alpha\geq 0$.
       Employing \eqref{pp} in \eqref{3b3}, we obtain
       $$ \hspace{-8cm}
       \sum_{n=0}^{\infty}\overline{p}_{4,8} \left( 384\cdot p^{2\alpha} n+8(2p^{2\alpha}-1)\right) q^n $$
       $$\hspace{3cm}\equiv \Big\{ \sum_{\substack{k=-(p-1)/2 \\ k\neq (\pm p-1)/6}
       }^{k=(p-1)/2}(-1)^kq^{(3k^2+k)/2}f\left( -q^{(3p^2+(6k+1)p)/2},-q^{(3p^2-(6k+1)p)/2}\right)   $$
       \begin{equation}\label{f1f2}
       +(-1)^{(\pm p-1)/6}q^{(p^2-1)/24}f_{p^2} \Big\}\pmod{4}.
       \end{equation}    
       Extracting the term involving $q^{{pn} + (p^2-1)/24}$ from  \eqref{f1f2}, dividing 
       by $q^{(p^2-1)/24}$ and then replacing $q^p$ by $q$, we obtain    
       \begin{equation}\label{3b4}
      \sum_{n=0}^{\infty}\overline{p}_{4,8} \left( 384\cdot p^{2\alpha+1} (n)+8(2p^{2(\alpha+1)}-1)
      \right) q^n \equiv f_p\pmod4.
       \end{equation}
       Extracting the terms involving $q^{pn}$ from \eqref{3b4} and replacing $q^p$ by $q$, we obtain
       \begin{equation}\label{3b5}
       \sum_{n=0}^{\infty}\overline{p}_{4,8} \left( 384\cdot p^{2(\alpha+1)} n+8(2p^{2(\alpha+1)}-1)\right) q^n \equiv f_1\pmod4.
       \end{equation}
       which is the $\alpha$ + 1 case of \eqref{t62}. Thus by principle of mathematical induction,the proof of \eqref{t62} is complete. Extracting the terms involving $q^{pn+j}$, for $ 1\leq j \leq p - 1$, from \eqref{3b4}, we arrive at \eqref{t63}.  
  \end{proof}
 
      \section{Congruences for $\overline{p}_{6,12}(n)$}    
  \begin{theorem}For all integers $n\geq 0$,$\alpha\geq 0$ and $1\le j \le(p-1)$, we have          
           \begin{equation}\label{t81}
           \sum_{n=0}^{\infty}\overline{p}_{6,12}\left( 24\cdot p^{2\alpha} n+p^{2\alpha})\right)  q^n \equiv f_1 \pmod{4},
           \end{equation}
           \begin{equation}\label{t82}
                {\overline{p}_{6,12}\left( p^{2\alpha+1}(24(pn+j)+p) \right) } \equiv 0\pmod{4}.
               \end{equation} 
            \end{theorem}
\begin{proof}
 Setting $j=6$ and $k=12$ in \eqref{p10}, we note that
 \begin{equation*}
 \sum_{n=0}^{\infty}\overline{p}_{6,12}(n)q^n=\dfrac{(-q;q)_{\infty}(q^6;q^{12})_{\infty}}{(q;q)_{\infty}(-q^6;q^{12})_{\infty}}.\end{equation*}
 Applying elementary $q$-operation and using \eqref{fff}, we obtain
 \begin{equation}\label{z1}
 \sum_{n=0}^{\infty}\overline{p}_{6,12}(n)q^n=\dfrac{f_2f_6^2f_{24}}{f_1^2f_{12}^3}.
 \end{equation}
 Employing \eqref{4} in \eqref{z1}, we obtain
 \begin{equation}\label{z2}
 \sum_{n=0}^{\infty}\overline{p}_{6,12}(n)q^n=\dfrac{f_6^2f_{24}}{f_{12}^3}\left( \dfrac{f_6^4f_9^6}{f_3^8f_{18}^3}+2q\dfrac{f_6^3f_9^3}{f_3^7}+4q^2\dfrac{f_6^2f_{18}^3}{f_3^6}\right) .
 \end{equation}
 Extracting the terms involving $q^{3n+1}$ from \eqref{z2}, we obtain
 \begin{equation}\label{z3}
 \sum_{n=0}^{\infty}\overline{p}_{6,12}(3n+1)q^n= 2\dfrac{f_2^5f_3^3f_8}{f_1^7f_4^3}.
 \end{equation}
 Employing \eqref{fp1} in \eqref{z3}, we obtain 
 \begin{equation}\label{z5}
 \sum_{n=0}^{\infty}\overline{p}_{6,12}(3n+1)q^n\equiv 2\dfrac{f_3^3}{f_1}\pmod{4}.
 \end{equation}
 Employing \eqref{5} in \eqref{z5}, we obtain
 \begin{equation}\label{z6}
 \sum_{n=0}^{\infty}\overline{p}_{6,12}(3n+1)q^n\equiv 2\dfrac{f_4^3f_6^2}{f_2^2f_{12}}+2q\dfrac{f_{12}^3}{f_4}\pmod{4}.
 \end{equation}
 Extracting the terms involving $q^{2n}$ from \eqref{z6} and using \eqref{fp1}, we obtain
 \begin{equation}\label{z7}
 \sum_{n=0}^{\infty}\overline{p}_{6,12}(6n+1)q^n \equiv 2 f_2^2\pmod{4}.
 \end{equation}
 Again, extracting the terms involving $q^{2n}$ from \eqref{z7} and using \eqref{fp1}, we obtain
 \begin{equation}\label{z8}
 \sum_{n=0}^{\infty}\overline{p}_{6,12}(12n+1)q^n\equiv  2f_2\pmod{4}.
 \end{equation}
 Again, extracting the terms involving $q^{2n}$ from \eqref{z8}, we obtain
 \begin{equation}\label{z9}
 \sum_{n=0}^{\infty}\overline{p}_{6,12}(24n+1)q^n\equiv  2f_1\pmod{4},
 \end{equation}
 which is the $\alpha = 0$ case of \eqref{t81}. Assume \eqref{t81} is true for any 
      $\alpha\geq 0$.
     Employing \eqref{pp} in \eqref{t81}, we obtain
     $$ \hspace{-8cm}
     \sum_{n=0}^\infty{\overline{p}_{6,12}\left( p^{2\alpha}(24n+ 1) \right) q^n} $$$$\hspace{3cm}\equiv \Big\{ \sum_{\substack{k=-(p-1)/2 \\ k\neq (\pm p-1)/6}
     }^{k=(p-1)/2}(-1)^kq^{(3k^2+k)/2}f\left( -q^{(3p^2+(6k+1)p)/2},-q^{(3p^2-(6k+1)p)/2}\right)
     $$
     \begin{equation}\label{f1f}
     +(-1)^{(\pm p-1)/6}q^{(p^2-1)/24}f_{p^2} \Big\}\pmod{4}.
     \end{equation}    
     Extracting the term involving $q^{{pn} + (p^2-1)/24}$ from  \eqref{f1f}, dividing 
     by $q^{(p^2-1)/24}$ and then replacing $q^p$ by $q$, we obtain    
     \begin{equation}\label{z10}
    \sum_{n=0}^\infty{\overline{p}_{6,12}\left( p^{2\alpha+1}(24n+p) \right) q^n} \equiv f_{p}\pmod{4}.
     \end{equation}
     Extracting the terms involving $q^{pn}$ from \eqref{z10} and replacing $q^p$ by $q$, we obtain
     \begin{equation}\label{z11}
      \sum_{n=0}^\infty{\overline{p}_{6,12}\left( p^{2(\alpha+1)}(24n+1) \right) q^n} \equiv f_{1}\pmod{4},\end{equation}
     which is the $\alpha$ + 1 case of \eqref{t81}. Thus by principle of mathematical induction,the proof of \eqref{t81} is complete.
    Extracting the terms involving $q^{pn+j}$, for $ 1\leq j \leq p - 1$, from \eqref{z10}, we arrive at \eqref{t82}.
     \end{proof}

 \begin{theorem}For all integers $n\geq 0$ and $\alpha\geq 0$, we have           
            \begin{equation}\label{t111}
            \sum_{n=0}^{\infty}\overline{p}_{6,12}\left( 3\cdot 2^{2\alpha} n+2^{2\alpha})\right)  q^n \equiv 2\dfrac{f_{3}^3}{f_1}\pmod{4},
            \end{equation}
            \begin{equation}\label{t112}
              \sum_{n=0}^{\infty}\overline{p}_{6,12}\left( 3\cdot 2^{2\alpha+1} n+2^{2(\alpha+1)})\right)  q^n \equiv 2\dfrac{f_{6}^3}{f_2}\pmod{4}, 
                \end{equation}
             \begin{equation}\label{t112}
                          \overline{p}_{6,12}\left( 3\cdot 2^{2(\alpha+1)}) n+5\cdot2^{2(\alpha+1)})\right) \equiv 0\pmod{4}. 
                            \end{equation}           
             \end{theorem}
 
 \begin{proof}      
From \eqref{z5}, we obtain 
 \begin{equation}\label{a10}
 \sum_{n=0}^{\infty}\overline{p}_{6,12}(3n+1)q^n\equiv  2\dfrac{f_{3}^3}{f_1}\pmod{4}.
 \end{equation} The remaining part of the proof is similar to proofs of the identies \eqref{t51}-\eqref{t53}.

\end{proof}
\begin{theorem}For all integers $n\geq 0$ and $\alpha\geq 0$, we have

  \begin{equation}\label{t121}
            \sum_{n=0}^{\infty}\overline{p}_{6,12}\left( 2^{2\alpha+1}(3n+1)\right)  q^n \equiv 4\dfrac{f_{3}^3}{f_1}\pmod{8},
            \end{equation}
            \begin{equation}\label{t122}
             \sum_{n=0}^{\infty}\overline{p}_{6,12}\left( 2^{2\alpha+2}(3n+2)\right)  q^n \equiv 4\dfrac{f_{6}^3}{f_2}\pmod{8}, 
                \end{equation}
             \begin{equation}\label{t123}
                          \overline{p}_{6,12}\left( 2^{2\alpha+2}(6n+5)\right)  \equiv 0\pmod{8}. 
                            \end{equation}          
\end{theorem}

\begin{proof}
 Extracting the terms involving $q^{3n+2}$ from \eqref{z2}, we obtain
 \begin{equation}\label{a131}
 \sum_{n=0}^{\infty}\overline{p}_{6,12}(3n+2)q^n=4\dfrac{f_2^4f_{6}^3f_8}{f_1^6f_4^3}.
 \end{equation}
 Employing \eqref{fp2} in \eqref{a131}, we obtain
\begin{equation}\label{a151}
 \sum_{n=0}^{\infty}\overline{p}_{6,12}(3n+2)q^n\equiv4\dfrac{f_{6}^3}{f_2}\pmod{8}.\end{equation} 
 Extracting the terms involving $q^{2n}$ from \eqref{a151}, we obtain 
 \begin{equation}\label{a161}
 \sum_{n=0}^{\infty}\overline{p}_{6,12}(6n+2)q^n\equiv4\dfrac{f_{3}^3}{f_1}\pmod{8}.\end{equation}
The remaining part of the proof is similar to proofs of the identies \eqref{t51}-\eqref{t53}.
\end{proof}

\begin{theorem}For all integers $n\geq 0$, $\alpha\geq 0$ and $1\le j\le(p-1)$, we have
 \begin{equation}\label{t131}
             \sum_{n=0}^{\infty}\overline{p}_{6,12}(2\cdot p^{2\alpha}(24n+1))q^n\equiv  4f_1\pmod{8},
            \end{equation}
            \begin{equation}\label{t132}
             \overline{p}_{6,12}(2\cdot p^{2\alpha+1}(24(pn+j)+1))\equiv  0\pmod{8}.
                \end{equation}
\end{theorem}

\begin{proof}
 Extracting the terms involving $q^{3n+2}$ from \eqref{z2}, we obtain
 \begin{equation}\label{a13}
 \sum_{n=0}^{\infty}\overline{p}_{6,12}(3n+2)q^n=4\dfrac{f_2^4f_{6}^3f_8}{f_1^6f_4^3}.
 \end{equation}
 Employing \eqref{fp2} in \eqref{a13}, we obtain
 \begin{equation}\label{a15}
 \sum_{n=0}^{\infty}\overline{p}_{6,12}(3n+2)q^n\equiv4\dfrac{f_{6}^3}{f_2}\pmod{8}.\end{equation}
 Extracting the terms involving $q^{2n}$ from \eqref{a15}, we obtain
 \begin{equation}\label{a16}
 \sum_{n=0}^{\infty}\overline{p}_{6,12}(6n+2)q^n\equiv4\dfrac{f_{3}^3}{f_1}\pmod{8}.\end{equation} 
 Employing \eqref{5} in \eqref{a16}, we obtain
 \begin{equation}\label{a17}
 \sum_{n=0}^{\infty}\overline{p}_{6,12}(6n+2)q^n\equiv4\left( \dfrac{f_4^3f_6^2}{f_2^2f_{12}}+q\dfrac{f_{12}^3}{f_4}\right) \pmod{8}.\end{equation} 
 Extracting the terms involving $q^{2n}$ from \eqref{a17} and using \eqref{fp1}, we obtain
 \begin{equation}\label{a18}
 \sum_{n=0}^{\infty}\overline{p}_{6,12}(12n+2)q^n \equiv 4 f_2^2\pmod{8}.
 \end{equation}
 Again extracting the terms involving $q^{2n}$ from \eqref{a18} and using \eqref{fp1}, we obtain
 \begin{equation}\label{a19}
 \sum_{n=0}^{\infty}\overline{p}_{6,12}(24n+2)q^n\equiv  4f_2\pmod{8}.
 \end{equation}
 Again, extracting the terms involving $q^{2n}$ from \eqref{a19}, we obtain
 \begin{equation}\label{a20}
 \sum_{n=0}^{\infty}\overline{p}_{6,12}(48n+2)q^n\equiv  4f_1\pmod{8}.
 \end{equation}
The rest of the proof is similar to the proof of the identities \eqref{t81}-\eqref{t82}.  
 \end{proof}

 \section{Congruences for $\overline{p}_{8,16}(n)$}
 \begin{theorem} Let $j \in \{0,2,3,4\}$ and $k \in \{0,1,3,4,5,6\}$ . Then for all integers $\alpha \geq0$ and $\beta \geq0$, we have 
  \begin{equation}\label{w1}
   \sum_{n=0}^{\infty}\overline{p}_{8,,16}\Big(8\cdot5^{2\alpha}\cdot7^{2\beta}(n)+3\cdot5^{2\alpha}\cdot7^{2\beta}\Big)q^n\equiv8f_1^9 \pmod{16}, 
   \end{equation}  
  \begin{equation}\label{w2}
  \sum_{n=0}^{\infty}\overline{p}_{8,16}\Big(8\cdot5^{2\alpha+1}\cdot7^{2\beta}(n)+7\cdot5^{2\alpha+1}\cdot7^{2\beta}\Big)q^n\equiv8qf_5^9 \pmod{16}, 
    \end{equation}
  \begin{equation}\label{w3}
      \overline{p}_{8,16}\Big(8\cdot5^{2\alpha+1}\cdot7^{2\beta}(5n+j)+7\cdot5^{2\alpha+1}\cdot7^{2\beta}\Big)\equiv0 \pmod{16}, 
     \end{equation}      
 \begin{equation}\label{w4}
       \sum_{n=0}^{\infty}\overline{p}_{8,16}\Big(8\cdot5^{2\alpha}\cdot7^{2\beta+1}(n)+5\cdot5^{2\alpha}\cdot7^{2\beta+1}\Big)q^n\equiv8q^2f_7^9 \pmod{16}, 
     \end{equation} 
  \begin{equation}\label{w5}
        \overline{p}_{8,16}\Big(8\cdot5^{2\alpha}\cdot7^{2\beta+1}(7n+k)+5\cdot5^{2\alpha}\cdot7^{2\beta+1}\Big)\equiv 0 \pmod{16}. 
      \end{equation}

  \end{theorem}
 \begin{proof}
 Setting $j=8$ and $k=16$ in \eqref{p10}, we note that
 	 \begin{equation*}
 	 \sum_{n=0}^{\infty}\overline{p}_{8,16}(n)q^n=\dfrac{(-q;q)_{\infty}(q^8;q^{16})_{\infty}}{(q;q)_{\infty}(-q^8;q^{16})_{\infty}}.\end{equation*}
 	 Applying elementary $q$-operation and using \eqref{fff}, we obtain
 	 \begin{equation}\label{b1}
 	 \sum_{n=0}^{\infty}\overline{p}_{8,16}(n)q^n=\dfrac{f_2f_8^2f_{32}}{f_1^2f_{16}^3}.
 	 \end{equation}
 	 Employing \eqref{1} in \eqref{b1}, we obtain
 	 \begin{equation}\label{b3}
 	 \sum_{n=0}^{\infty}\overline{p}_{8,16}(n)q^n= \dfrac{f_8^7f_{32}}{f_2^4f_{16}^5}+2q\dfrac{f_4^2f_8f_{32}}{f_2^4f_{16}}.
 	 \end{equation}
 	 Extracting the terms involving $q^{2n+1}$ from \eqref{b3}, we obtain
 	 \begin{equation}\label{b4}
 	\sum_{n=0}^{\infty}\overline{p}_{8,16}(2n+1)q^n=2\dfrac{f_2^2f_4f_{16}}{f_1^4f_{8}}.
 	 \end{equation}
 	 Employing \eqref{2} in \eqref{b4}, we obtain
 	 \begin{equation}\label{b5}
 	 \sum_{n=0}^{\infty}\overline{p}_{8,16}(2n+1)q^n=2\dfrac{f_2^2f_4f_{16}}{f_{8}}\left( \dfrac{f_4^{14}}{f_2^{14}f_{8}^4}+4q\dfrac{f_4^2f_{8}^4}{f_2^{10}}\right).
 	 \end{equation}
 	 Extracting the terms involving $q^{2n+1}$ from \eqref{b5}, we obtain 	 
 	 \begin{equation}\label{b7}
 	  	 \sum_{n=0}^{\infty}\overline{p}_{8,16}(4n+3)q^n=8\dfrac{f_2^3f_4^3f_{8}}{f_1^{8}}.
 	  	 \end{equation}
 	   Employing \eqref{fp1} in \eqref{b7}, we obtain 	  	 
 	  	  	  	 \begin{equation}\label{b9}
 	  	  	  	  	  	  	  	 \sum_{n=0}^{\infty}\overline{p}_{8,16}(4n+3)q^n=8f_2^9\pmod{16}.
 	  	  	  	  	  	  	  	 \end{equation}
 Extracting the terms involving $q^{2n}$ from \eqref{b9}, we obtain	  	  	  	  	  	  	  	 
 \begin{equation}\label{b10}
 \sum_{n=0}^{\infty}\overline{p}_{8,16}(8n+3)q^n=8f_1^9\pmod{16}.
 \end{equation}	  	  	  	  	  	  	  	 
  The rest of the proof is similar to proofs of the identies \eqref{y1}-\eqref{y5}.
 	 \end{proof}

 \end{document}